\documentclass[11pt]{article}

\usepackage[margin=1in]{geometry}
\usepackage{graphicx}
\usepackage{amsmath, amsfonts, amsthm, amssymb}
\usepackage{verbatim}
\usepackage[utf8]{inputenc}
\usepackage{xcolor} 
\usepackage{authblk} 
\usepackage[symbol]{footmisc} 
\usepackage{pdfpages}

\newtheorem{thm}{Theorem}[section]

\newtheorem{lem}[thm]{Lemma}
\newtheorem{cor}[thm]{Corollary}
\newtheorem{conj}[thm]{Conjecture} 
\theoremstyle{definition}
\newtheorem{definition}[thm]{Definition}

\renewcommand*{\thefootnote}{\fnsymbol{footnote}}
\renewcommand*{\thefootnote}{\arabic{footnote}}
\newcommand\blfootnote[1]{%
  \begingroup
  \renewcommand\thefootnote{}\footnote{#1}%
  \addtocounter{footnote}{-1}%
  \endgroup
}

\title{An Algebraic-Coding Equivalence to the Maximum Distance Separable Conjecture \blfootnote{Support from the National Science Foundation, the Department of Mathematics, University of Michigan; and the Undergraduate Research Opportunity Program at the University of Michigan is acknowledged.}}
\author{Steven Damelin\footnote{Department of Mathematics, University of Michigan, Ann Arbor, MI 48109, USA;  damelin@umich.edu (corresponding author)}, Daniel Kaiser\footnote{Department of Mathematics, University of Michigan, Ann Arbor, MI 48109, USA; dankai@umich.edu}, Jeffrey Sun\footnote{Department of Mathematics, University of Michigan, Ann Arbor, MI 48109, USA; jeffjeff@umich.edu}, Safal Bora\footnote{Department of Mathematics, University of Michigan, Ann Arbor, MI 48109, USA; safal@umich.edu}}

\begin{document}
\maketitle

\begin{abstract}
\indent \par In this paper, we provide Algebraic-Coding necessary and sufficient conditions for the Maximum Distance Separable Conjecture to hold.
\end{abstract}

\section{Introduction: The MDS Conjecture}
\indent \par Let \textit{k} be an integer such that $2 \leq k \leq q = p^r,$ where \textit{p} is prime and \textit{r} is a positive integer. A $k\times n$ maximum distance separable ($k\times n$ MDS) code $M$ is a $k \times n$ matrix with entries in $\mathbb{F}_q$ such that every set of $k$ columns of $M$ is linearly independent. \indent \par The Maximum Distance Separable (MDS) conjecture is a well-known problem in coding theory and algebraic geometry with important consequences for example to the study of arcs in finite projective spaces \cite{Ball1} and to coding theory \cite{Ball3, Ball2}.

The conjecture, first posed by Singleton in 1964 \cite{Singleton} gives a possible upper-bound on the size of a $k\times n$ MDS code. 
More precisely, the MDS conjecture says the following:

\begin{conj}
The maximum width, $n,$ of a $k \times n$ MDS code with entries in $\mathbb{F}_q$ is $q+1$, unless $q$ is even and $k \in \{3,q-1\}$, in which case the maximum width is $q+2$.
\end{conj}

We remark that there exist $k\times n$ MDS codes that attain the maximum possible width as given by the MDS conjecture. These are the Reed-Solomon codes defined in Definition 1.3 (e). See \cite{Ball2}.

\medskip

Henceforth all matrices will have entries in $\mathbb{F}_q$. When we speak to linear combinations, we mean nontrivial $\mathbb{F}_q$-linear combinations.

\subsection{Some known work on the MDS Conjecture}

We provide some known (not exhaustive) work on the MDS Conjecture. An excellent survey on the MDS conjecture  can be found in \cite{Ball1}. See in addition the references \cite{Ball, De Beule, S.Ball, Ball2, Ball3}. The MDS conjecture has been verified when $q$ is prime, see \cite{S.Ball}. It is also known to hold when $q$ is a square and with $k\leq c\sqrt{pq}$ where the constant $c$ depends on whether $q$ is odd or even. When $q$ not a square, it is known to be true for $k\leq c'\sqrt{pq}$ where $c'$ depends on whether $q$ is odd or even.  See \cite{Ball1}. It is also known to hold for all $k\times n$ MDS codes with alphabets of size at most 8.

This paper is motivated by the following strongest result on the MDS conjecture which currently holds.
\begin{thm}[\cite{S.Ball}] 
Let k be an integer such that $2 \leq k \leq q = p^r,$ where p is prime and r is a positive integer. The MDS Conjecture is true whenever $k \leq 2p-2.$
\end{thm}

We are ready to state the main result of the paper which provides necessary and sufficient conditions for the MDS conjecture to hold. This is given in Theorem 1.4 below.

\subsection{Statement of  Main Result}

In this section, we state our main result, Theorem 1.4.  We need some definitions and notation. These are given in:

\begin{definition}
\begin{itemize}
\item[(a)] We denote by $\mathcal{P}_q$ the ring $\mathbb{F}_q[x]/(x^q-x)$ of polynomial functions over $\mathbb{F}_q$ of maximum degree $q$. $\mathcal{P}_q$ is a vector space over $\mathbb{F}_q$. 
\item[(b)] Throughout, the \textit{perp space} of a vector, $\vec{v}$ in an inner product space, $(\mathbf{V},\cdot):=\mathbf{V}$  is the set given by $\vec{v}^\perp = \{ \vec{w} \in \mathbf{V} | \vec{w} \cdot \vec{v} = 0 \}$.\\
	\indent The \textit{perp space} of a subspace, $\mathbf{U} \subseteq \mathbf{V}$, where $\mathbf{V}$ is an inner product space is the set given by $\mathbf{U}^\perp = \{ \vec{w} \in \mathbf{V} | \forall \vec{u} \in \mathbf{U},\vec{w} \cdot \vec{u} = 0 \}$. ($\vec{w}$ and $\vec{u}$ are orthogonal). 
\item[(c)] For every non-negative integer $n$, we define the subset $\mathcal{O}_n \subset \mathcal{P}_q$ as the set of polynomials in $\mathcal{P}_q$ that are either the zero polynomial, or have at most $n$ distinct roots in $\mathbb{F}_q$. If $n \geq q$, then $\mathcal{O}_n = \mathcal{P}_q$.
\item[(d)] We denote by $\langle Y, Z \rangle$ the subspace of $\mathcal{P}_q$ generated by the elements of subspaces $Y$ and $Z$.
\item[(e)] A \textit{Reed-Solomon code} of dimension $k$ is a $k \times q$ matrix with entries in $\mathbb{F}_q$ such that each column of the matrix is of the form $[1,a,a^2,\ldots,a^{k-2},a^{k-1}]^\intercal$ for some $a \in \mathbb{F}_q$.
\item[(f)] An \textit{Extended Reed-Solomon code} is a Reed-Solomon code with the column $[0,0,0,\ldots,0,1]^\intercal$ appended.
\end{itemize}
\end{definition}

Both Reed-Solomon and Extended Reed-Solomon codes are $k\times n$ MDS codes, see \cite{Ball1} and it is proved in \cite{Seroussi}, that for odd $q$,  no column other than $[0,0,0,\ldots,0,1]^\intercal$ can be appended to a Reed-Solomon code to produce another $k\times n$ MDS code.\\
\medskip

We have
\begin{thm}

Let k be an integer such that $2 \leq k \leq q = p^r,$ where p is prime and r is a positive integer. Suppose that either
		\begin{enumerate}
			\item $q$ odd
			\item $q$ even and $k\in \{3, q-1\}$.
		\end{enumerate}
		
Consider the following statements (1-5) below:

\begin{enumerate}		

\item[(1)] The MDS Conjecture is true. That is, the maximum width, $n,$ of a $k \times n$ MDS code with entries in $\mathbb{F}_q$ is $q+1$, unless $q$ is even and $k \in \{3,q-1\}$, in which case the maximum width is $q+2$.

\item[(2)] Let $M'$ be a $k \times (q+2)$ matrix. Then some linear combination of the rows of $M'$ has at least $k$ zero entries.
		
\item[(3)] Let $M'$ be a $k \times (q+2)$ matrix such that the first two columns of $M'$ are $[1,0,\ldots,0]^\intercal$ and $[0,1,0,\ldots,0]^\intercal$. Then some linear combination of the rows of $M'$ has at least $k$ zero entries.

\item[(4)] There do not exist distinct subspaces $Y$ and $Z$ of $\mathcal{P}_q$ such that
		\begin{enumerate}
			\item $dim(\langle Y, Z \rangle) = k$.
			\item $dim(Y) = dim(Z) = k-1$.
			\item $\langle Y, Z \rangle \subset \mathcal{O}_{k-1}$
			\item $Y\cup Z \subset \mathcal{O}_{k-2}$
			\item $Y\cap Z \subset \mathcal{O}_{k-3}$.
		\end{enumerate}
		
\item[(5)] There is no integer $s$ with $k < s \leq q$ such that the Reed-Solomon code $\mathcal{R}$ with entries in $\mathbb{F}_q$ of dimension $s$ can have $s-k+2$ columns $\mathcal{B} = \{b_1,\ldots,b_{s-k+2}\}$ added to it, such that:
		\begin{enumerate}
			\item Any $s \times s$ submatrix of $\mathcal{R} \cup \mathcal{B}$ containing the first $s-k$ columns of $\mathcal{B}$ is independent (non-zero determinant).
			\item $\mathcal{B} \cup \{[0,0,\ldots,0,1]^{\intercal} \}$ is independent.
		\end{enumerate}
		
	\end{enumerate}
\bigskip	

Then the following holds true:

\begin{itemize}
\item[(Part A)]:  (1), (2), (3) and (5) are equivalent.
\item[(Part B)]:  (3) implies (4).
\end{itemize}
\end{thm}

\section{The Proof of Theorem 1.4}

In this section, we prove Theorem 1.4.

\subsection{Prerequisites to the proof of Theorem 1.4}

In preparation for the proof of Theorem 1.4, we need various prerequisites. All of these are well known facts in algebra and can be found for example in  \cite{A}. 

\begin{definition}
Suppose $M$ is a $k\times n$ matrix with entries in some field $\mathbf{F}$. Suppose $\vec{L} = [l_1,l_2,\ldots l_k]$ $\epsilon$ $\mathbf{V} \setminus \{ \vec{0} \} $ where $\mathbf{V}$ is a vector space of dimension $k$ over $\mathbf{F}$. Then $\vec{L}M$ is a $1\times n$ vector referred to as a \textit{linear combination of the rows of $M$}.
\end{definition}

\begin{lem}
	A polynomial in $\mathbb{F}_q^s$ has a field element, $a$, as a root if and only if it is orthogonal to the vector $\vec{r}_a = [1, a, a^2, a^3 \ldots a^{s-1}]$. Here, a polynomial belongs to $\mathbb{F}_q^s$ via sending $c_{s-1}x^{s-1}+...+c_1x+c_0$ to $(c_0,c_1,...,c_{s-1})\in \mathbb{F}_q^s$
\end{lem}

\begin{cor}
	The set of polynomials with $k$ distinct roots is precisely the set of polynomials orthogonal to the span($\vec{r}_{a_1}, \ldots, \vec{r}_{a_k}$), for some choice of $k$ field elements.
\end{cor}

\begin{lem}
A set of polynomials does not have $k$ distinct roots if and only if it does not intersect $( \vec{r}_{a_1}, \ldots, \vec{r}_{a_k} )^\perp$ for any choice of $k$ field elements.
\end{lem}

\subsection{Proof of Theorem 1.4}
\noindent We now present the proof of Theorem 1.4
\begin{proof}
\subsubsection*{\underline{$(1)\iff(2)$}}
Consider the statement  that $q+1$ is the maximum width $n$ of a $k \times n$ matrix $M$, such that every set of $k$ columns of $M$ is linearly independent. This is clearly logically equivalent to the statement (i), "That there is no $k \times (q+2)$ matrix $M'$ such that every set of $k$ columns of $M'$ is linearly independent." Then (i) is equivalent to the statement (ii), ''For every $M'$ a
$k\times (q + 2)$ matrix, some $k$ columns of $M'$ form a $k\times k$ matrix that is singular.''

Let $A$ be a $k \times k$ submatrix of $M'$. Then let $\vec{r}_i$ denote the $i^{\text{th}}$ row of $M'$ and let $\vec{r}_{i,A}$ denote the $i^{\text{th}}$ row of $A$. Then each $\vec{r}_{\ell,A}$ consists of the entries of $\vec{r}_\ell$ within the submatrix $A$.

The statement  (ii) is that ''for every $M'$ a
$k\times (q + 2)$ matrix, some $k$ columns of $M'$ form a $k\times k$ matrix that is singular.'' This is logically equivalent to (iii), "For some $k \times k$ submatrix $A$ of $M'$, and some set $a_i$ of coefficients in $\mathbb{F}_q$, not all zero, it is true that $\sum_i a_i\vec{r}_{i,A} = \vec{0}$." But to say that $\sum_i a_i\vec{r}_{i,A} = \vec{0}$ is to say that the $j$-th entry $\left(\sum_i a_i\vec{r}_i\right)_j$ of the row vector $\sum_i a_i\vec{r}_i$ is zero whenever the $j$-th column of $M'$ is part of the submatrix $A$. Therefore, (iii) is logically equivalent to the statement (iv), "There exists a set $\mathcal{A}$ of $k$ columns of $M'$, and a nontrivial linear combination $\vec{r} = \sum_i a_i\vec{r}_i$ of the rows of $M'$, such that the entry of $\vec{r}$ at each column in $\mathcal{A}$ is zero." That is, that for some nontrivial linear combination $\vec{r} = \sum_i a_i\vec{r}_i$ of the rows of $M'$, $\vec{r}$ has at least $k$ zero entries.

\subsubsection*{\underline{$(2)\iff(3)$}}
First, we separate the right $k \times q$ submatrix $T$ of $M'$ from the two leftmost columns, which we call $\vec{v}$ and $\vec{w}$. 
$$
M' = \underset{ ^{k \times (q+2)}}{\begin{bmatrix}
	\vec{v} & \vec{w} & \vert & T &
	\end{bmatrix}}
$$
We can always reduce to the case where the rows of $T$ are linearly independent. Indeed, if the rows of $T$ are linearly dependent, then there exists a nontrivial linear combination of the rows of $M'$ with $q$ zeroes. 
If $\vec{v}$ and $\vec{w}$ are linear independent, we can always left-multiply $M'$  by some invertible matrix that takes $\vec{v}$ to $[1,0,0,\ldots,0]^\intercal$ and $\vec{w}$ to $[0,1,0,\ldots,0]^\intercal$. Such a matrix exists due to the assumption that $\vec{v}$ and $\vec{w}$ are linearly independent and thus can be completed to form a basis of $\mathbb{F}_q^k,$ and likewise for $[1,0,0,\ldots,0]^\intercal, [0,1,0,\ldots,0]^\intercal;$ hence, a change-of-basis matrix between these two completed bases will provide an invertible matrix with the desired mapping of $\vec{v},\vec{w}$. We perform such a left-multiplication on $M'$, and henceforth may assume that $M'$ is of the form
$$
M' = \underset{ ^{k \times (q+2)}}{\begin{bmatrix}
	[1,0,0,\ldots,0]^\intercal & [0,1,0,\ldots,0]^\intercal & \vert & T &
	\end{bmatrix}}
$$
Note that a  $k \times k$ submatrix of $M'$ being singular is unaffected by left-multiplication by an invertible $k \times k$ matrix, since the determinant is multiplicative. (left-multiplication by an invertible $k\times k$ matrix takes invertible $k \times k$ submatrices to invertible $k \times k$ submatrices, and singular $k \times k$ submatrices to singular $k \times k$ submatrices).

If $\vec{v}$ and $\vec{w}$ are linearly dependent, we claim that any $k \times k$ submatrix of $M'$ containing $\vec{v}$ and $\vec{w}$, say $S$, is singular.  If our claim is true, we need not worry about the case when 
$\vec{v}$ and $\vec{w}$ are dependent.

To prove our claim, we argue as follows: The existence of such a $k\times k$ submatrix $S$ may occur because $k \leq q < q+2$.  $S$, containing two dependent columns, may be acted upon by column operations to produce a column of all 0s. By transposition and row operations, this column may be placed in the top row of some reduced matrix of S, say S'. Then, the determinant of S' via expansion of cofactors must be 0, hence, S' is singular. Now since the determinant is multiplicative, we may write $(S')^\intercal = ES$, where E is the matrix product of elementary row and column operations. It follows that $0 = det(S') = det((S')^\intercal) = det(ES) = det(E)det(S) = det(S)$, and hence S is indeed singular.\\  

\subsubsection*{\underline{$(3)\Rightarrow(4)$}}
Suppose $M'$ is a matrix of the form
$$
M' = \underset{ ^{k \times (q+2)}}{\begin{bmatrix}
	[1,0,0,\ldots,0]^\intercal & [0,1,0,\ldots,0]^\intercal & \vert & T &
	\end{bmatrix}}
$$
with the rows of $T$ linearly independent. Let $\vec{r}_i$ denote the $i^{\text{th}}$ row of $M'$ and $\vec{t}_i$ denote the $i^{\text{th}}$ row of the submatrix $T$. Observe that this implies that each $\vec{r}_i$ is of the form $[\delta_{1,i}, \delta_{2,i} | \vec{t}_i]$, where $\delta_{i,j}$ is the Kronecker delta function.

The condition $(3)$ is now that, without any further conditions on $M'$, there must be a nontrivial linear combination $\vec{r} = \sum_i a_i\vec{r}_i$ of the rows $\vec{r}_i$ of $M'$ that has $k$ zero entries.

Let $\vec{r} = \sum_i a_i\vec{r}_i$, and let $\vec{t} = \sum_i a_i\vec{t}_i$, which is $\vec{r}$ with the first two entries dropped. Then because $\vec{v} = [1,0,\ldots,0]^\intercal$ and $\vec{w} = [0,1,0,\ldots,0]^\intercal$, we have $(\vec{r})_1 = a_1$, and $(\vec{r})_2 = a_2$, where $(\vec{r})_j$ denotes the $j^{\text{th}}$ entry of $\vec{r}$. What is relevant is that $(\vec{r})_1 = 0$ if and only if $a_1 = 0$, and $(\vec{r})_2 = 0$ if and only if $a_2 = 0$.

Therefore, $\vec{r}$ has at most $k-1$ zero entries if and only if one of the following hold:
\begin{enumerate}
\item $\vec{t}$ has at most $k-1$ zero entries and $a_1,a_2 \neq 0$.
\item $\vec{t}$ has at most $k-2$ zero entries if either $a_1 = 0$ or $a_2 = 0$, but not both.
\item $\vec{t}$ has at most $k-3$ zero entries if $a_1 = a_2 = 0$.
\end{enumerate}

The condition (3) then becomes (v), defined to be that for some nontrivial linear combination $\vec{r}_0 = \sum_i a_i\vec{r}_i$, and $\vec{t}_0 = \sum_i a_i\vec{t}_i$, one of the following hold:
\begin{enumerate}
\item $\vec{t}_0$ has at least $k$ zero entries.
\item $\vec{t}_0$ has at least $k-1$ zero entries if either $a_1 = 0$ or $a_2 = 0$.
\item $\vec{t}_0$ has at least $k-2$ zero entries if $a_1 = a_2 = 0$.
\end{enumerate}

We proceed by identifying the $\vec{t}_i$ with elements of $\mathcal{P}_q$ in the following way. We begin by choosing arbitrary elements $\alpha_i \in \mathbb{F}_q$ with $i \neq j \Rightarrow \alpha_i \neq \alpha_j$, and identify column \textit{j} of submatrix \textit{T} with the element $\alpha_j$. Then we identify each row $\vec{t}_i$ with the unique function $f_i$ in $\mathcal{P}_q$ defined by $f_i(\alpha_j) = (\vec{t}_i)_j$ for all $1 \leq j \leq q$.

Then the linear combinations of the $\vec{t}_i$, tuples in $\mathbb{F}_q^q$, correspond (via a canonical isomorphism to the polynomial space over $\mathbb{F}_q$) to the elements of the \textit{k}-dimensional vector space in $\mathcal{P}_q$ spanned by the $f_i$. Furthermore, an entry $(\vec{t}_i)_j$ is zero if and only if $f_i(\alpha_j)$ is zero. We then translate the above condition.

Let $\vec{r} = \sum_i a_i\vec{r}_i$ and $f = \sum_i a_if_i$. Then $\vec{r}$ has at most $k-1$ zero entries if and only if one of the following hold:
\begin{enumerate}
\item $f$ has at most $k-1$ roots.
\item $f$ has at most $k-2$ roots and either $a_1 = 0$ or $a_2 = 0$, but not both.
\item $f$ has at most $k-3$ roots and $a_1 = a_2 = 0$.
\end{enumerate}

The condition (v) is equivalent to (vi), defined to be that for some nontrivial linear combination $\vec{r}_0 = \sum_i a_i\vec{r}_i$, and $f_0 = \sum_i a_if_i$, one of the following hold:
\begin{enumerate}
\item $f_0$ has at least $k$ roots.
\item $f_0$ has at least $k-1$ roots and either $a_1 = 0$ or $a_2 = 0$, but not both.
\item $f_0$ has at least $k-2$ roots and $a_1 = a_2 = 0$.
\end{enumerate}

Let $Y_T$ and $Z_T$ be the vector spaces $Y_T = \langle\{f_i | i\neq 1\}\rangle$ and $Z_T = \langle\{f_i | i\neq 2\}\rangle$.  Then $Y_T$ and $Z_T$ are (\textit{k}-1)-dimensional vector spaces  and each of $Y_T$ and $Z_T$ is generated by $k-1$ rows. Similarly, $\langle Y_T, Z_T \rangle$ is a \textit{k}-dimensional vector space.

Then (vi) is equivalent to (vii), defined to be that one of the following hold: 
\begin{enumerate}
\item $\langle Y_T, Z_T \rangle \not\subset \mathcal{O}_{k-1}$.
\item $Y_T \not\subset \mathcal{O}_{k-2}$.
\item $Z_T \not\subset \mathcal{O}_{k-2}$.
\item $Y_T\cap Z_T \not\subset \mathcal{O}_{k-3}$.
\end{enumerate}

To see this we observe that  if $f_0$ has $k$ roots and $a_1,a_2 \neq 0$, then this means that $f_0 = \sum_{i=1}^q a_if_i \in \langle \{f_i\} \rangle = \langle Y_T, Z_T \rangle$ has at least $k$ roots.  Then $\langle Y_T,Z_T \rangle \not\subset \mathcal{O}_{k-1}$.  Similarly, if $f_0$ has at least $k-2$ roots and $a_1 = a_2 = 0$ then  $f_0 = \sum_{i = 3}^q a_if_i \in \langle \{f_i : i\neq1,2 \} \rangle = Y_T \cap Z_T$ has at least $k-2$ roots.
Then $Y_T \cap Z_T \not\subset \mathcal{O}_{k-3}$. The cases  $f_0$ has at least $k-1$ roots and $a_1=0\neq a_2$ and $f_0$ has at least $k-1$ roots and $a_2 = 0 \neq a_1$ 
can be handled similarly and give $Y_T \not\subset \mathcal{O}_{k-2}$ and $Z_T \not\subset \mathcal{O}_{k-2}$ .
 \indent Having proven (vi) is equivalent to (vii), it now remains to show that any subspaces $Y_0$ and $Z_0$ satisfying both
\begin{enumerate}
\item $dim(\langle Y_0, Z_0 \rangle) = k$ 
\item $dim(Y_0) = dim(Z_0) = k-1$
\end{enumerate}
can be constructed in this way, as $Y_0 = Y_T$, and $Z_0 = Z_T$, for some $k \times q$ matrix $T$.

Fix such subspaces $Y_0$ and $Z_0$. Recall that
\[
dim(Y_0) + dim(Z_0) = dim(Y_0\cap Z_0) + dim(\langle Y_0, Z_0\rangle)
\] Thus,
\[
2k - 2 = dim(Y_0\cap Z_0) + k
\]
\indent\indent so that
\[
dim(Y_0\cap Z_0) = k-2.
\]

Thus we can find a set $\mathcal{N}$ of $k-2$ linearly independent vectors that span $Y_0 \cap Z_0$. Then we can find a vector $\vec{y} \in Y_0$ and $\vec{z} \in Z_0$, such that $\langle \vec{y}, \mathcal{N} \rangle = Y_0$, and $\langle \vec{z}, \mathcal{N} \rangle = Z_0$. Then, arranging the elements of $\mathcal{N}$ as a $q \times k-2$ matrix $N$, we set
\[
T = \begin{bmatrix}
\vec{z} & \vec{y} & \vert & N &
\end{bmatrix}^\intercal
\]
so that $Y_0 = Y_T$ and $Z_0 = Z_T$.

Then since any subspaces $Y$ and $Z$ of $\mathcal{P}_q$ satisfying $dim(\langle Y_0, Z_0 \rangle) = k$ and $dim(Y_0) = dim(Z_0) = k-1$ can be expressed as $Y_T$ and $Z_T$ for some $T$, (vii) is equivalent to the statement (viii), defined to be that for any such subspaces $Y$ and $Z$, we must have either
\begin{enumerate}
\item $\langle Y, Z \rangle \not\subset \mathcal{O}_{k-1}$, or
\item $Y \not\subset \mathcal{O}_{k-2}$, or
\item $Z \not\subset \mathcal{O}_{k-2}$, or
\item $Y\cap Z \not\subset \mathcal{O}_{k-3}$.
\end{enumerate}

That is, (4) which says that we cannot have subspaces $Y$ and $Z$ of $\mathcal{P}_q$ satisfying:

\begin{enumerate}
\item $dim(\langle Y, Z \rangle) = k$
\item $dim(Y) = dim(Z) = k-1$.
\item $\langle Y, Z \rangle \subset \mathcal{O}_{k-1}$
\item $Y, Z \subset \mathcal{O}_{k-2}$
\item $Y\cap Z \subset \mathcal{O}_{k-3}$.
\end{enumerate}

\subsubsection*{\underline{$(1)\iff(5)$}}

\indent Suppose we have a $k\times q+2$ MDS code $M'$, which is conjectured to be impossible. We will provide a construction whose existence will be equivalent to the failure of the MDS conjecture.  (Recall the proof of (3) implies (4)).

$$
M' = \underset{ ^{k \times q+2}}{\begin{bmatrix}
\vec{v} & \vec{w} & \vert & M &
\end{bmatrix}}
$$

First, we separate the right $k \times q$ submatrix $M$ of $M'$ from its additional columns $\vec{v}$ and $\vec{w}$.

We have that no nontrivial linear combination of the rows of $M'$ have $k$ zeroes. Splitting into three cases:
\begin{itemize}
\item No nontrivial linear combination of the rows of $M$ has $k$ zeroes.
\item Row combinations taking $\vec{v}$ or $\vec{w}$ to the zero vector never take $k-1$ columns of $M$ to the zero vector.
\item Row combinations taking both $\vec{v}$ and $\vec{w}$ to the zero vector never take $k-2$ columns of $M$ to the zero vector.
\end{itemize}

Observe that, after labelling the columns of $M$ with the elements of $\mathbb{F}_q$, the rows of $M$ can be viewed as the value sets of polynomials over $\mathbb{F}_q$. Indeed, they define unique polynomials of degree $\leq q-1$. That is, they can be viewed as vectors in the space $\mathbb{F}_q^q$ of polynomials of degree $\leq q-1$ over $\mathbb{F}_q$.

Due to linearity, linearly combining the rows of $M$ cover the $k$-dimensional space that the rows span.

Furthermore, a zero in a nontrivial linear combination of the rows of $M$ represents a root of
the polynomial represented by the linear combination.

We define $s$ such that the highest degree of a polynomial represented by a row of $M$ is $s-1$. Then, the rows of $M$ span a $k$-dimensional subspace of the space of polynomials of degree $\leq s-1$, which is isomorphic to $\mathbb{F}_q^s$.

The MDS conjecture is known to be true whenever $s=k$. Furthermore, only two $k\times q+1$ MDS codes are known when $s \neq k$. See \cite{Ball1}. 

We now have the the following:

There exists a $k$-dimensional subspace $X$ of $\mathbb{F}_q^s$, and distinct $(k-1)$-dimensional subspaces $Y, Z \subset X$ such that:
\begin{itemize}
\item No polynomial in $X$ has $k$ distinct roots.
\item No polynomial in $Y$ or $Z$ has $k-1$ distinct roots.
\item No polynomial in $Y \cap Z$ has $k-2$ distinct roots.
\item Some polynomial in $X$ has a nonzero $x^{s-1}$ term.
\end{itemize}

The subspace $Y$ respectively $Z$ arises from the observation that row combinations taking $\vec{v}$ respectively $\vec{w}$ to the zero vector form a $(k-1)$-dimensional subspace of the nontrivial linear combinations of the rows of $M'$, and row combinations taking both $\vec{v}$ and $\vec{w}$ to the zero vector are a $(k-2)$-dimensional subspace ($Y \cap Z$ has dimension $k-2$).

Assuming the polynomials are coordinatized coefficient-wise, with the $k^{\text{th}}$ entry the coefficient of degree $(k-1)$, there is a simple condition for a polynomial having a certain root in $\mathbb{F}_q$.
(See Lemma 2.2 and Corollary 2.3).

Thus, $X$ does not intersect $\langle \vec{r}_{a_1}, \ldots, \vec{r}_{a_k} \rangle^\perp$ for any choice of $k$ field elements.

Furthermore, $X$ is not contained within the perp space $[0,0,\ldots,1]^\perp$, because, there is some polynomial in $X$ with nonzero degree-$(s-1)$ coefficient.

The independent condition can be translated
\begin{itemize}
\item $X$ is disjoint from all perp spaces $\langle \vec{r}_{a_1}, \ldots, \vec{r}_{a_k} \rangle^\perp$, for all choices of $k$ field elements.
\item $Y$ and $Z$ are disjoint from all perp spaces $\langle \vec{r}_{a_1}, \ldots, \vec{r}_{a_{k-1}} \rangle^\perp$, for all choices of $k-1$ field elements.
\item $Y \cap Z$ are disjoint from all perp spaces $\langle \vec{r}_{a_1}, \ldots, \vec{r}_{a_{k-2}} \rangle^\perp$, for all choices of $k-2$ field elements.
\item $X \not\subseteq [0,0,\ldots,1]^\perp$.
\end{itemize}

Observe that all the conditions of disjointness are between spaces whose dimension add up to the whole dimension of the space. By non-degeneracy of the dot product as a bilinear form, the perp spaces of the disjoint spaces are disjoint. Also, $X \not\subseteq [0,0,\ldots,1]^\perp$ is equivalent to $[0,0,\ldots,1] \not\in X^\perp$.
 After taking duals \footnote{Given a vector space $V$, the corresponding dual vector space $V*$ consists of all linear forms on $V$ together with the vector space structure of pointwise addition and scalar multiplication by constants.}, the condition becomes:
\begin{itemize}
\item $X^\perp$ an $(s-k)$-dimensional space, is disjoint from $\langle \vec{r}_{a_1}, \ldots, \vec{r}_{a_k} \rangle$, for all such choices of $k$ field elements.
\item $Y^\perp$ and $Z^\perp$ (similarly), $(s-k+1)$-dimensional spaces containing $X^\perp$, are disjoint from $\langle \vec{r}_{a_1}, \ldots, \vec{r}_{a_{k-1}} \rangle$, for all such choices of $k-1$ field elements.
\item $\langle Y^\perp Z^\perp \rangle$, an $(s-k+2)$-dimensional space, is disjoint from $\langle \vec{r}_{a_1}, \ldots, \vec{r}_{a_{k-2}} \rangle$, for all such choices of $k-2$ field elements.
\item $[0,0,\ldots,1] \not\in X^\perp$.
\end{itemize}

Observe that the $\vec{r}_a$, viewed as columns of a matrix, form a Reed-Solomon code, $\mathcal{R}$.

$Y$ and $Z$, as defined, exist if and only if independent vectors $\vec{y}$ and $\vec{z}$ exist s.t.  $Y=\langle X,  \vec{y} \rangle$, $Z=\langle X, \vec{z}\rangle$,
 with $X$ a subspace of dimension $s-k$. Let ${\cal B}$ be a basis of $X$ and $\vec{p}=[0,0,...,1]$.

Then can write a new $s \times (q+s+3-k)$ matrix,
$$
\underset{ ^{s \times q+3+s-k}}{\begin{bmatrix}
\vec{p} & \vert & \vec{y} & \vert & \vec{z} & \vert & \mathcal{B} & \vert & \mathcal{R}
\end{bmatrix}}
$$
and translate to conditions on this matrix:

\begin{itemize}
\item $\{\mathcal{R} \cup \vec{p}\}$ is an extended Reed-Solomon Code.
\item Any $s \times s$ submatrix containing $\mathcal{B}$ but not $\vec{p}$ is invertible.
\item $\mathcal{B}$ does not span $\vec{p}$. \footnote{This condition involving $\vec{p}$ is not needed to state, but is included by us because the MDS conjecture has apriori been settled for the case $s=k$, see \cite{Ball1}.}\\
\newline\indent This proves $(1)\iff(5)$, and hence the proof of Theorem 4.1 is complete.
\end{itemize}

\end{proof}





\end{document}